\documentclass{amsart}
\usepackage{amssymb,latexsym}
\theoremstyle{plain}
\newtheorem{theorem}{Theorem}

\newtheorem{lemma}{Lemma}
\theoremstyle{definition}

\newtheorem{remark}{Remark}

\begin{document}

\title[symmetric rank]
{Gaps in the pairs (border rank,\ symmetric rank) for symmetric tensors}
\author{E. Ballico}
\address{Dept. of Mathematics\\
 University of Trento\\
38123 Povo (TN), Italy}
\email{edoardo.ballico@unitn.it}
\thanks{The author was partially supported by MIUR and GNSAGA of INdAM (Italy).}

\begin{abstract}
Fix integers $m \ge 2$, $s\ge 5$ and $d\ge 2s+2$. Here we describe the possible symmetric tensor ranks $\le 2d+s-7$
of all symmetric tensors (or homogeneous degree $d$ polynomials) in $m+1$
variables with border rank $s$.
\end{abstract}

\maketitle

\section{Introduction}\label{S1}

An important practical question concerning symmetric tensors (e.g. in Signal Processing, Statistics
and Data Analysis) is their ``~minimal~'' decomposition as a sum
of pure symmetric tensors (see e.g. \cite{cm}, \cite{ls}, \cite{bcmt}, \cite{cglm}, \cite{lt}, \cite{bgi}, \cite{l} and references
therein). This problem may be translated
into the following problem for homogeneous polynomials in $m+1$ variables: for any degree $d$ homogeneous polynomial
$f\in \mathbb {K}[x_0,\dots ,x_m]$ find the minimal integer $r$ such that $f = \sum _{i=1}^{r} {L_i}^d$, where each $L_i$ is a homogeneous degree $1$ polynomial. The latter problem is translated in the following way
into a problem concerning Veronese embeddings of $\mathbb {P}^m$.

Let $\nu _d: \mathbb {P}^m \to \mathbb {P}^{n_{m,d}}$, $n_{m,d}:= \binom{m+d}{m}-1$, denote the degree $d$ Veronese embedding of $\mathbb
{P}^m$. Set $X_{m,d}:= \nu _d(\mathbb {P}^m)$. We often write $n$ instead of $n_{m,d}$. For any subset or closed subscheme $A$ of a projective space $\mathbb {P}^k$ let
$\langle A\rangle$ denote its linear span. For any integer $s>0$ the $s$-secant variety
$\sigma _s(X_{m,d})$ is the closure in $\mathbb {P}^n$ of the union of all linear spaces spanned
by $s$ points of $X_{m,d}$. Fix $P\in \mathbb {P}^n$. The symmetric rank $sr (P)$ of $P$ is the minimal
cardinality of a finite set $S\subset X_{m,d}$ such that $P\in \langle S\rangle$. The border rank $br (P)$ of $P$
is the minimal integer $s>0$ such that $P\in \sigma _s(X_{m,d})$. There
is another notion of rank of $P$ (the cactus rank $cr (P$ (\cite{bb+}, \cite{bk}), but
we do not need to define it, because in the range $br (P) \le d+1$ we always have $cr (P)=br (P)$
(Remark \ref{v2}). For any fixed $s\ge 2$ one would like to have
the stratification by the symmetric rank of $\sigma _s(X_{m,d})\setminus \sigma _{s-1}(X_{m,d})$, i.e. to know
what are the ranks of the homogeneous degree $d$ polynomials with border rank $s$. This is due to Sylvester if $m=1$, i.e. for binary forms
(\cite{cs}, \cite{lt}, Theorem 4.1, \cite{bgi}). For general $m$ this is known if $s=2,3$ (\cite{bgi}) and if $s=4$ (\cite{bb}).  For all positive integers $a, b$ set
$$\sigma _{a,b}(X_{m,d}):= \{P\in \mathbb {P}^n: br (P)=a, \ sr (P) = b\}.$$
Notice that $\sigma _{a,b}(X_{m,d})= \{P\in \sigma _a(X_{m,d})\setminus \sigma _{a-1}(X_{m,d}): sr (P)=b\}$
if $a\ge 2$, that $\sigma _{a,a}(X_{m,d})$ contains a non-empty
open subset of $\sigma _a(X_{m,d})$ if $\sigma _{a-1}(X_{m,d}) \ne \mathbb {P}^n$ and that $\sigma _{a,b}(X_{m,d}) =\emptyset$ if $b<a$. If either $m=1$ and $s$ is very
low ($s=2,3$ in \cite{bgi}, $s=4$ in \cite{bb}),
then for for fixed $s$ and large $d$
near $s$ several integers are not the symmetric rank of any $P\in \sigma _s(X_{m,d})\setminus \sigma _{s-1}(X_{m,d})$.
Here we show that this is the case for arbitrary $s$, $d$ not too small, but
for low ranks, i.e. if we assume $r\le 2d+s-7$. We prove the following result.

\begin{theorem}\label{i1}
Fix integers $m, s, d, r$ such that $m\ge 2$, $s\ge 5$, $d \ge 2s+2$ and $s \le r\le 2d+s-7$. 

Then $\sigma _{s,r}(X_{m,d}) \ne \emptyset$ if and only if one of the following conditions is satisfied:
\begin{itemize}
\item  $r=s$;
\item $d+2-s \le r \le d+s-2$ and $r+s \equiv d \pmod{2}$;
\item $2d+2-s \le r \le 2d+s-7$.
\end{itemize}\end{theorem}

If $\sigma _s(X_{m,d}) \supsetneq \sigma _{s-1}(X_{m,d})$, then a general $P\in \sigma _s(X_{m,d})$ satisfies $br (P) = s = sr (P)$. Hence in the set-up of Theorem \ref{i1}
only the pairs $(s,r)$ with $2\le s < r$ need to be checked. The statement of Theorem \ref{i1} is of the form ``~if and only if~''. However, the proofs of both
implications use similar tools. The key technical tool used in almost all our lemmas is an inductive method to handle cohomology groups (vanishing and non-vanishing)
often called the Horace Method. The starting observation is that for every $s$ as in Theorem \ref{i1} there is a zero-dimensional
scheme $A\subset \mathbb {P}^m$ such that $\deg (A)=s$ and $P\in \langle \nu _d(A)\rangle$ (Remark \ref{v2}). Moreover if $sr (P) > br (P)$, then there
is a finite set $B\subset \mathbb {P}^m$ such that $\sharp (B) = sr (P)$, $P\in \langle \nu _d(A)\rangle \cap \langle \nu _d(B)\rangle$ and the scheme $A\cup B$
has two strong properties: $h^1(\mathcal {I}_{A\cup B}(d)) >0$ and there is a line or a conic (say $D$) such that $\deg ((A\cup B)\cap D) \ge \deg (D)\cdot d+2$ and $B\setminus B\cap D = A\setminus A\cap D$ (Lemmas \ref{c0} and \ref{v4}). In this way it is easy to get the non-existence part of Theorem \ref{i1}. The existence part in the range $s+r\le d+2$ (i.e. with $r=d+2-s$) is done
taking $A$ and $B$ contained in a line $L$. To cover the case $r+s \equiv d \pmod{2}$ and $d+2-s < r \le d+s-2$ we take $A = (A\cap L) \sqcup E$ and $B = (B\cap L)\sqcup 
E$ with $E\subset \mathbb {P}^m\setminus L$, $\sharp (E) = (s+r-d-2)/2$, $B\cap A\cap L=\emptyset$ and $\deg (A\cap L) + \sharp (B\cap L) =d+2$. To cover the case
in which $r+s$ is even and $2d+2 \le r+s$ we use a smooth conic $C \subset \mathbb {P}^m$ and take $A = (A\cap C)\sqcup E$, $B = (B\cap C) \sqcup E$ with $A\cap B \cap C = \emptyset$, $\deg (A\cap C) + \sharp (B\cap C) =2d+2$ and $\sharp (E) = (r+s-2d-2)/2$. To cover the case $2d+3-s \le r \le 2d+s-7$ and $r+s$ odd
we use a reducible conic instead of $C$; we need two different constructions according to the parity of the integer $(2d+3+s-r)/2 $ (Lemmas \ref{o3} and \ref{o4}).

In all cases the delicate part is the proof that there is no finite set $S\subset \mathbb {P}^m$ such that $P\in \langle \nu _d(S)\rangle$
and $\sharp (S)<\sharp (B) =r$. In all cases we again use Lemmas \ref{c0} and \ref{v4}.

We work over an algebraically closed base field $\mathbb {K}$ such that $\mbox{char}(\mathbb {K})=0$.

\section{The proof}

For any sheaf $\mathcal {F}$ on $\mathbb {P}^m$ and any integer $i\ge 0$ set $h^i(\mathcal {F}) :=
\dim (H^i(\mathbb {P}^m,\mathcal {F}))$. For any scheme $X$, any effective Cartier divisor $D$
of $X$ and any closed subscheme $Z\subset X$ let $\mbox{Res}_D(Z)$ denote the residual scheme
of $Z$ with respect to $D$, i.e. the closed subscheme of $X$ with $\mathcal {I}_Z:\mathcal {I}_D$
as its ideal sheaf. For any $R\in \mbox{Pic}(X)$ we have the following exact sequence
of coherent sheaves (called the {\it residual exact sequence}):
\begin{equation}\label{eqc0}
0 \to \mathcal {I}_{\mbox{Res}_D(Z)}\otimes R(-D) \to \mathcal {I}_Z\otimes R\to \mathcal {I}_{D\cap Z,D}\otimes (R\vert D)\to 0
\end{equation}

We need the following lemma (see \cite{c2} for the case in which the scheme $Z$ is reduced,
\cite{bgi}, Lemma 34, for the case $z\le 2d+1$,
and \cite{ep} for a strong tool to prove much more in $\mathbb {P}^2$).

Fix positive integers $m, d$, any $P\in \mathbb {P}^{n_{m,d}}$ and any finite set $B\subset \mathbb {P}^m$. We say that $B$ evinces $sr (P)$ if $P\in \langle \nu _d(B)\rangle$
and $\sharp (B) =sr (P)$.

\begin{lemma}\label{c0}
Fix integers $m, d, z$ such that $m\ge 2$ and $0 < z < 3d$. Let $Z\subset \mathbb {P}^m$ be a zero-dimensional scheme
such that $\deg (Z)=z$. If $m>2$, then assume $\deg (Z) -\deg (Z_{red}) \le d$. We have $h^1(\mathcal {I}_Z(d)) > 0$ if and only if either
there is a line $L\subset \mathbb {P}^m$ such that $\deg (L\cap Z)\ge d+2$ or there
is a conic $T\subset \mathbb {P}^m$ such that $\deg (T\cap Z) \ge 2d+2$.
\end{lemma}

\begin{proof}
Since $Z$ is zero-dimensional, the restriction map
$H^0(Z,\mathcal {O}_Z(d)) \to H^0(W,\mathcal {O}_W(d))$ is
surjective for any $W\subseteq Z$. Hence for any $W\subseteq Z$ we have
$h^1(\mathcal {I}_W(d)) \le h^1(\mathcal {I}_Z(d))$. Since $h^0(L,\mathcal {O}_L(d)) =d+1$
for any line $L$ and $h^0(T,\mathcal {O}_T(d)) =2d+1$ for any conic $T$, we get
the ``~if~'' part.
Now assume $h^1(\mathcal {I}_Z(d)) > 0$. 

\quad (a) First assume $m=2$. Apply \cite{ep}, Remarques (i) at page 116.

\quad (b) Now assume $m\ge 3$. We use induction on $m$. Let $H_1\subset \mathbb {P}^m$ be
a hyperplane such that $\deg (H_1\cap Z)$ is maximal.
Set $Z_0:=Z$, $Z_1:= \mbox{Res}_{H_1}(Z_0)$ and $w_1:=
\deg (Z_0\cap H_1)$. As in the
proof of \cite{bb}, Proposition 12, we define recursively the hyperplanes $H_i\subset \mathbb {P}^m$, $i\ge 2$,
the schemes $Z_i\subseteq Z_{i-1}$, and the integers $w_i$, $i\ge 1$, in the following way. Let $H_i$ be any hyperplane such that
$\deg (Z_{i-1}\cap H_i)$ is maximal. Set $Z_i:= \mbox{Res}_{H_i}(Z_{i-1})$ and $w_i:= \deg (H_i\cap Z_{i-1})$.
Any zero-dimensional scheme $F\subset \mathbb {P}^m$ with $\deg (F) \le m$
is contained in a hyperplane. Hence if $w_i \le m-1$, then $w_{i+1} = 0$ and $Z_i = \emptyset$. Since $z<3d$, we get $w_i=0$
for all $i\ge d$ and $Z_d =\emptyset$. For any integer $i > 0$ the residual
sequence (\ref{eqc0})
gives the following exact sequence:
\begin{equation}\label{eqc1}
0 \to \mathcal {I}_{Z_i}(d-i) \to \mathcal {I}_{Z_{i-1}}(d-i+1) \to \mathcal {I}_{H_i\cap Z_{i-1},H_i}(d-i+1) \to 0
\end{equation}
Since $h^1(\mathcal {I}_Z(d)) >0$ and $Z_d =\emptyset$, (\ref{eqc1}) gives
the existence of an integer $x$ such that $1\le x \le d-1$ and $h^1(H_x,\mathcal {I}_{H_x\cap Z_{x-1},H_x}(d-x+1)) >0$. We call $e$ the minimal such an integer $x$. First assume $e=1$, i.e. assume
$h^1(H_1,\mathcal {I}_{Z\cap H_1}(d)) >0$. Since $\deg (Z\cap H_1) \le \deg (Z) < 3d$, the inductive assumption on $m$ gives that either there is a line $L\subset H_1$ such that $\deg (L\cap Z)\ge 2$ or there
is a conic $T\subset H_1$ such that $\deg (T\cap Z) \ge 2d+2$. From now on we assume $e\ge 2$. First assume $w_e \ge 2(d-e+1)+2$.
Since $w_i \ge w_e$ for all $i<e$, we get $z \ge 2e(d-e+1) +2e$. 
Since $2 \le e \le d-1$ and $z<3d$, we get a contradiction. Hence
$w_e \le 2(d-e+1)+1$. Since
$h^1(H_e,\mathcal {I}_{H_e\cap Z_{e-1},H_e}(d-e+1)) >0$ and $w_e\le 2(d-e+1)+1$, there is a line $L
\subset H_e$ such that $\deg (L\cap Z_{e-1}) \ge d-e+3$
(\cite{bgi}, Lemma 34). Since $Z_{e-1}\ne \emptyset$,
$Z_{e-2}$ spans $\mathbb {P}^m$. Hence there is a hyperplane
$M\subset \mathbb {P}^m$ such that $M\supset L$ and $\deg (M\cap Z_{e-2}) \ge
\deg (Z_{e-2}\cap L)+m-2 \ge d-e+m+1$. Hence $w_i\ge d-e+m+1$ for all $i<e$.
Hence $z \ge e(d-e+3)+(e-1)(m-2)$. 

First assume $e\ge 3$. Since $3d >
z \ge e(d-e+3)+(e-1)(m-2)$ and $3\le e \le d$, we get a contradiction.

Now assume $e=2$.  We have $\deg (L\cap Z)\ge d+1$. If $\deg (L\cap Z)\ge d+2$, then
we are done. Hence we may assume $\deg (L\cap Z)=d+1$. Set $W_0:= Z$. Let $M_1\subset \mathbb {P}^m$ be a hyperplane
containing $L$ and with $m_1:= \deg (M_1\cap W_0)$ maximal among the hyperplanes
containing $L$. We define recursively the hyperplanes $M_i\subset \mathbb {P}^m$, $i\ge 2$,
the schemes $W_i\subseteq W_{i-1}$, and the integers $m_i$, $i\ge 1$, in the following way. Let $M_i$ be any hyperplane such that
$\deg (W_{i-1}\cap M_i)$ is maximal. Set $W_i:= \mbox{Res}_{M_i}(W_{i-1})$ and $m_i:= \deg (H_i\cap W_{i-1})$. We have $m_i\le m_{i-1}$ for all $i$, $m_1\ge \deg (L\cap Z)+m-2
\ge d+m-2$ and $m_i=0$ if $m_{i-1} \le m-1$. As above there is a minimal integer $f$ such
$1 \le f \le d-1$ and $h^1(M_f,\mathcal {I}_{M_f\cap W_{f-1},M_f}(d-f+1))>0$.
As above we get a contradiction, unless $f=2$. 
Assume $f=2$. Since $m_2\le z/2 \le 2(d-1)+1$, there is a line $D\subset M_2$
such that $\deg (D\cap W_1) \ge d+1$. Let $E$ be any connected component
of $Z$. If $E$ is reduced, then $L\cap \mbox{Res}_{M_1}(E) =\emptyset$, because
$M_1\supset L$.
If $E$ is not reduced, then $\deg (M_1\cap \mbox{Res}_{M_1}(E)) \le \deg (E\cap M_1)$,
because $\mbox{Res}_{M_1}(E) \subseteq E$. Since $\deg (Z)-\deg (Z_{red}) \le d$, we get
$D\ne L$. Assume for the moment that either $D\cap L \ne \emptyset$ or $m\ge 4$, i.e. assume
the existence of a hyperplane of
$\mathbb {P}^m$ containing $D\cup L$. Hence $w_1\ge 2d+1$. Hence
$\deg (Z_1) \le \deg (W_1) \le d$. Hence $h^1(\mathcal {I}_{Z_1}(d-1)) =0$. Hence $h^1(H_2,\mathcal {I}_{Z_1\cap H_2}(d-1))=0$, contradicting the assumption $e=2$. Now assume $m=3$ and $D\cap L=\emptyset$. We may also assume $\deg (L\cap Z) = \deg (D\cap Z) = d+1$. Let $N\subset \mathbb {P}^3$
be a general quadric surface containing $D\cup L$. The quadric surface $N$ is smooth. Since $\deg (\mbox{Res}_N(Z)) \le z-2d-2 \le d-1$,
we have $h^1(\mathcal {I}_{\mbox{Res}_N(Z)}(d-2))=0$. Hence the exact sequence
\begin{equation}\label{eqc2}
0 \to \mathcal {I}_{\mbox{Res}_N(Z)}(d-2) \to \mathcal {I}_Z(d) \to \mathcal {I}_{Z\cap N,N}(d) \to 0
\end{equation}
gives $h^1(N,\mathcal {I}_{Z\cap N,N}(d))>0$. Since $D\cap L=\emptyset$, $D$ and $L$ belong
to the same ruling of $N$, say $D, L\in \vert \mathcal {O}_N(1,0)\vert$. Since
$\deg (Z\cap L) =\deg (Z\cap D)=d+1$, we have $h^i(N,\mathcal {I}_{Z\cap N,N}(d,d))
= h^i(N,\mathcal {I}_{\mbox{Res}_{D\cup L}(Z\cap N),N}(d-2,d))$, $i=0,1$.
Since $\deg (\mbox{Res}_{D\cup L}(Z\cap N)) = \deg (Z\cap N) -2d-2 \le d-1$,
we have $h^1(N,\mathcal {I}_{\mbox{Res}_{D\cup L}(Z\cap N),N}(d-2,d)) =0$, a contradiction.
\end{proof}

We recall the following result (\cite{bb1}, Lemma 1).

\begin{lemma}\label{v1} Fix $P\in \mathbb{P}^n$. Assume
the existence of zero-dimensional schemes $A, B\subset \mathbb {P}^m$ such
that $A\ne B$, $P\in \langle \nu _d(A)\rangle \cap \langle \nu _d(B)\rangle$, $P\notin \langle \nu _d(A')\rangle$ for any $A'\subsetneq A$ and $P\notin \langle \nu _d(B')\rangle$ for any $B'\subsetneq B$.
Then $h^1(\mathcal {I}_{A\cup B}(d)) >0$.
\end{lemma}

\begin{remark}\label{v2}
Fix integers $m \ge 1$, $d \ge 2$ and $P\in \mathbb {P}^n$ such that $br (P) \le d+1$. By \cite{bgl},
Lemma 2.1.5 and Lemma 2.4.4, there is a smoothable 
zero-dimensional and Gorenstein scheme $A\subset \mathbb {P}^m$ such that $\deg (A) =br (P)$,  $P\in \langle \nu _d(A) \rangle$ and $P\notin \langle \nu _d(A') \rangle$ for any $A'\subsetneq A$. We will say that $A$ evinces $br (P)$. In this range the smoothable rank and the border rank coincide. Now
assume $br (P) \le (d+1)/2$. Using Lemma \ref{v1} and the inequality $2s \le d+1$ we get
that $A$ is the unique zero-dimensional scheme $E\subset \mathbb {P}^m$ such that
$P\in \langle \nu _d(E)\rangle$ and $\deg (E)\le s$. The uniqueness
of $A$ implies that $A$ also evinces the cactus rank $cr (P)$ of $P$. In particular $cr (P) =br (P)$
if $br (P) \le (d+1)/2$.
\end{remark}

\begin{lemma}\label{v3}
Fix a proper linear subspace $L$ of $\mathbb {P}^m$, an integer $d \ge 2$ and a finite
set $E\subset \mathbb {P}^m\setminus L$ such that $\sharp (E)\le d$. Then
$\dim (\langle \nu _d(E\cup L)\rangle ) =\dim (\langle \nu _d( L)\rangle )+\sharp (E)$.
For any closed subscheme $U\subseteq L$ we have $\langle \nu_d(U\cup E)\rangle \cap \langle \nu _d(L)\rangle
= \langle \nu _d(U)\rangle$.
For any $O\in \langle \nu _d(L\cup E)\rangle \setminus \langle \nu _d(E)\rangle$, the
set $\langle \{O\}\cup \nu _d(E)\rangle \cap \langle \nu _d(L)\rangle$ is a unique point.
\end{lemma}

\begin{proof}
Since $E$ is a finite set and $E\cap L=\emptyset$, a general hyperplane $H$ containing $L$ contains
no point of $E$. Since $E \cap H = \emptyset$, we have $\mathcal {I}_{E\cup H}(d)
\cong \mathcal {I}_E(d-1)$. Since $\sharp (E)\le d$, we
have $h^1(\mathcal {I}_E(d-1))=0$.  Hence $\dim (\langle \nu _d(H\cup E)\rangle )
=\dim (\langle \nu _d(H)\rangle )+\sharp (E)$. Since $L\subseteq H$, we get $\dim (\langle \nu _d(E\cup L)\rangle ) =\dim (\langle \nu_d( L)\rangle )+\sharp (E)$. Grassmann's formula
give $\langle \nu _d(L)\rangle \cap \langle \nu _d(E)\rangle =\emptyset$. Hence $\langle \nu_d(U\cup E)\rangle \cap \langle \nu _d(L)\rangle
= \langle \nu _d(U)\rangle$ for any $U\subseteq L$. Fix any $O\in \langle \nu _d(L\cup E)\rangle \setminus \langle \nu _d(E)\rangle$. Since $O\notin \langle \nu _d(E)\rangle$, we
have $\dim (\langle \{O\}\cup \nu _d(E)\rangle )= \dim (\langle \nu _d(E))+1$.
Since $O\in \langle \nu _d(L\cup E)\rangle$ and $\langle \nu _d(L)\rangle \cap \langle \nu _d(E)\rangle =\emptyset$, Grassmann's formula gives that $\langle \{O\}\cup \nu _d(E)\rangle \cap \langle \nu _d(L)\rangle$ is a unique point.
\end{proof}

In the same way we get the following result.

\begin{lemma}\label{v3.0}
Fix a conic $T \subset \mathbb {P}^m$, an integer $d \ge 5$ and a finite
set $E\subset \mathbb {P}^m\setminus T$ such that $\sharp (E)\le d-1$. Then
$\dim (\langle \nu _d(E\cup T)\rangle ) =\dim (\langle \nu _d( T)\rangle )+\sharp (E)$.
For any closed subscheme $U\subseteq T$ we have $\langle \nu_d(U\cup E)\rangle \cap \langle \nu _d(T)\rangle
= \langle \nu _d(U)\rangle$.
For any $O\in \langle \nu _d(T\cup E)\rangle \setminus \langle \nu _d(E)\rangle$, the
set $\langle \{O\}\cup \nu _d(E)\rangle \cap \langle \nu _d(T)\rangle$ is a unique point.
\end{lemma}

The following lemma was proved (with $D$ a hyperplane) in \cite{bb2}, Lemma 8. The same proof works for an arbitrary hypersurface $D$ of $\mathbb {P}^m$ (see also Remark \ref{v4.0} below).

\begin{lemma}\label{v4}
 Fix $P\in \mathbb{P}^n$. Assume
the existence of zero-dimensional schemes $A, B\subset \mathbb {P}^m$ such
that $A\ne B$, $P\in \langle \nu _d(A)\rangle \cap \langle \nu _d(B)\rangle$, $P\notin \langle \nu _d(A')\rangle$ for any $A'\subsetneq A$ and $P\notin \langle \nu _d(B')\rangle$ for any $B'\subsetneq B$.
Assume that $B$ is reduced. Assume the existence of a positive integer $t\le d$ and of a degree $t$ hypersurface $D\subset \mathbb {P}^m$
such that $h^1(\mathcal {I}_{\mbox{Res}_D(A\cup B)}(d-t)) =0$. Set $E:= B\setminus B\cap D$. Then $\nu _d(E)$ is linearly independent, $E = \mbox{Res}_D(A)$ and every unreduced connected component
of $A$ is contained in $D$. The linear space $\langle \nu _d(A)\rangle \cap \langle \nu _d(B)\rangle$
is the linear span of its supplementary subspaces $\langle \nu _d(E)\rangle$
and $\langle \nu _d(A\cap D)\rangle \cap \langle \nu _d(B\cap D)\rangle$.\end{lemma}

\begin{remark}\label{v4.0}
Take the set-up of Lemma \ref{v4}.

\quad {\emph {Claim :}} We have $\langle \nu _d(A\cap D)\rangle \cap \langle \nu _d(B\cap D)\rangle \ne \emptyset$ and there is $Q\in \langle \nu _d(A\cap D)\rangle \cap \langle \nu _d(B\cap D)\rangle$ such that $P\in \langle \{Q\}\cup \nu _d(E)\rangle$.

\quad {\emph {Proof of the Claim :}} Lemma \ref{v1}
gives $h^1(\mathcal {I}_{A\cup B}(d)) >0$. Since $P\notin \langle \nu _d(A')\rangle$ for any $A'\subsetneq A$ and $P\notin \langle \nu _d(B')\rangle$ for any $B'\subsetneq B$, we get $E\ne A$ and $E\ne B$, i.e.
$A\cap D\ne \emptyset$ and $B\cap D\ne \emptyset$. The residual exact sequence
(\ref{eqc0}) gives the following exact sequence:
\begin{equation}\label{eqv2}
0 \to \mathcal {I}_{\mbox{Res}_D(A\cup B)}(d-t) \to \mathcal {I}_{A\cup B}(d) \to \mathcal {I}_{(A\cup B)\cap, D,D}(d) \to 0
\end{equation}
From (\ref{eqv2}) and the definition of $E$ we get the last assertion of
Lemma \ref{v4}. Since $P\in \langle \nu _d(A)\rangle \cap \langle \nu _d(B)\rangle$, $P\notin \langle \nu _d(A')\rangle$ for any $A'\subsetneq A$ and $P\notin \langle \nu _d(B')\rangle$ for any $B'\subsetneq B$, $\nu _d(A)$ and $\nu _d(B)$ are linearly independent.
Since $h^1(\mathcal {I}_{A\cup B}(d)) >0$ (Lemma \ref{v1}), the exact sequence (\ref{eqv2})
gives $h^1(D,\mathcal {I}_{(A\cup B)\cap, D,D}(d))>0$. Hence the linear independence
of $\nu _d(A)$ and $\nu _d(B)$ implies $\langle \nu _d(A\cap D)\rangle \cap \langle \nu _d(B\cap D)\rangle \ne \emptyset$. Since $\nu _d(A)$ is linearly independent, $P\in \langle \nu _d(A)\rangle \cap \langle \nu _d(B)\rangle$, $P\notin \langle \nu _d(A')\rangle$ for any $A'\subsetneq A$ and $P\notin \langle \nu _d(B')\rangle$ for any $B'\subsetneq B$, the sets $\langle \{P\}\cup \nu _d(E)\rangle
\cap \langle \nu _d(A\cap D)\rangle$ and $\langle \{P\}\cup \nu _d(E)\rangle
\cap \langle \nu _d(A\cap D)\rangle$ are given by a unique point. Call it $Q_A$ and $Q_B$, respectively.
Obviously $P\in \langle \nu _d(E)\cup \{Q_A\}\rangle \cap \langle \nu _d(E)\cup \{Q_B\}\rangle$.
Since $P\notin \langle \nu _d(E)\rangle$, we get $Q_A = Q_B$. Set $Q:= Q_A$.
\end{remark}

For any reduced projective set $Y\subset \mathbb {P}^r$ spanning $\mathbb {P}^r$ and any $P\in \mathbb {P}^r$ let $r_Y(P)$ denote the minimal cardinality of a finite set $S\subset Y$ such
that $P\in \langle S\rangle$. The positive integer $r_Y(P)$ is often called the $Y$-rank of $P$.

\begin{lemma}\label{o1}
Assume $m=2$. Fix integers $w\ge 3$ and $d\ge 4w-1$. Take lines $L_1, L_2\subset \mathbb {P}^2$ such that $L_1\ne L_2$.
Set $\{O\}:= L_1\cap L_2$. Let $A_1\subset L_1$ be the degree $w$ effective divisor
of $L_1$ with $O$ as its support. Fix $O_2\in L_2\setminus \{O\}$
and let $A_2\subset L_2$ the degree $w$ effective divisor of $L_2$ with $O_2$ as its reduction. Set $A:= A_1\cup A_2$. Fix
$P\in \langle \nu _d(A)\rangle$ such that $P\notin \langle \nu _d(A')\rangle$ for any $A'\subsetneq A$. Then
$br (P) = 2w$ and $sr (P)
\ge  2d+3-2w$. There is $P$ as above with $sr (P) = 2d+3-2w$.
\end{lemma}

\begin{proof}
Since $\deg (A) = 2w \le (d+1)/2$ and $P\notin \langle \nu _d(A')\rangle$ for any $A'\subsetneq A$,
$A$ is the only scheme evincing $br (P)$ (Remark \ref{v2}).

\quad (a) In this step we prove the existence of $P$ as above and with $sr (P) \le 2d+3-2w$. Fix
$B_1\subset L_1\setminus \{O\}$ such that $\sharp (B_1) = d-w+2$ and $B_2\subset L_2\setminus \{P_2\}$ such
that $\sharp (B_2)=d-w+1$. Since $O\notin B_1$, $\deg (A_1\cup B_1)=d+2$ and $\nu _d(L_1)$ is
a degree $d$ rational normal curve in its linear span, the set $\langle \nu _d(A_1)\rangle\cap \langle \nu _d(B_1)\rangle$
is a single point (call it $P'$). We have $P'\notin \langle W\rangle$ if either $W\subsetneq A_1$ or $W\subsetneq B_1$.
Since $\dim (\langle \nu _d(L_1\cup L_2)\rangle )=2d$, $\nu _d(A_1\cup A_2\cup B_1\cup B_2)$
spans $\langle \nu _d(L_1\cup L_2)\rangle$ and $\dim (\langle \nu _d(B_1\cup
B_2)\rangle )=2d+2-2w$, the set $\langle \nu _d(A_1\cup
A_2)\rangle \cap \langle \nu _d(B_1\cup
B_2)\rangle $ is a line, $M$. Obviously $P'\in M$. Take as $P$ any of the points of $M\setminus \{P'\}$.
Since $P\in \langle \nu _d(B_1\cup B_2)\rangle$, we have $sr (P)\le 2d+3-2w$.

\quad (b) To conclude the proof it is sufficient to prove that $sr (P) \ge 2d+3-2w$. Assume $sr (P) \le 2d+2-2w$
and fix $B\subset \mathbb {P}^2$ evincing $sr (P)$. We have $h^1(\mathcal {I}_{A\cup B}(d)) >0$
(Lemma \ref{v1}) and $\deg (A\cup B) \le 2d+2$. Hence either there is a line $D\subset \mathbb {P}^2$
such that $\deg (D\cap (A\cup B)) \ge d+2$ or there is a conic $T$ such that
$\deg (T\cap (A\cup B)) \ge 2d+2$ (Lemma \ref{c0}). 

\quad (b1) Assume  the existence of a line $D$ such that $\deg (D\cap (A\cup B)) \ge d+2$. Since $\deg (\mbox{Res}_D(A\cup B))  \le (2d-2w+2)+2w-d-2 =d$,
we have $h^1(\mathcal {I}_{\mbox{Res}_D(A\cup B)}(d-1))=0$. Hence Lemma \ref{v4} gives
$A\subset D$, a contradiction.

\quad (b2) Assume the existence of a conic $T$ such that $\deg (T\cap (A\cup B)) \ge 2d+2$.
Since $\deg (A) +\deg (B) \le 2d+2$, we get $A\cap B = \emptyset$,
$\sharp (B) = 2d+2-2w$ and $A\cup B \subset T$.
We have $\deg (L_2\cap A) = w+1 $ and $\deg (L_1\cap A)=w$.
Since $w\ge 3$, Bezout theorem implies that $L_1\cup L_2$ is the unique conic containing $A$. Hence $T = L_1\cup L_2$.
Set $B_i:= B\cap L_i$. First assume $\sharp (B_1) \ge d-w+2$. Since $\sharp (B_1) \le d-w$, the
scheme
$\mbox{Res}_{L_1}(A\cup B)
= A_2\cup B_2$ has degree $\le d$. Hence $h^1(\mathcal {I}_{\mbox{Res}_{L_1}(A\cup B)}(d-1))=0$. Lemma \ref{v4} gives $A\subset L_1$, a contradiction. Now assume
$\sharp (B_1) \le d-w+1$. Let $G\subset L_1$ be the degree $w-1$ effective divisor
of $L_1$ with $O$ as its support. Since $\mbox{Res}_{L_2}(A\cup B) = G\cup B_1$ has degree
$\le d$, Lemma \ref{v4} implies $A\subset L_2$, a contradiction.
\end{proof}

\begin{lemma}\label{o2}
Assume $m=2$. Fix integers  $w\ge 2$ and $d\ge 4w+1$. Take lines $L_1, L_2\subset \mathbb {P}^2$ such that $L_1\ne L_2$.
Set $\{O\}:= L_1\cap L_2$. Let $A_1\subset L_1$ be the degree $w+1$ effective divisor
of $L_1$ with $O$ as its support. Fix $O_2\in L_2\setminus \{O\}$
and let $A_2\subset L_2$ the degree $w$ effective divisor of $L_2$ with $O_2$ as its reduction. Set $A:= A_1\cup A_2$. Fix
$P\in \langle \nu _d(A)\rangle$ such that $P\notin \langle \nu _d(A')\rangle$ for any $A'\subsetneq A$. Then
$br (P) = 2w+1$ and $sr (P)
\ge  2d+2-2w$. There is $P$ as above with $sr (P) = 2d+2-2w$.
\end{lemma}

\begin{proof}
Copy the proof of Lemma \ref{o1}. In step (b2) we have
$T = L_1\cup L_2$, because $\deg (L_1\cap A)=\deg (L_2\cap A) =w+1\ge 3$).
\end{proof}

\begin{lemma}\label{o3}
Fix integers $w\ge 3$, $s\ge 2w$ and assume $d\ge 2s-1$. Fix lines $D, R\subset \mathbb {P}^m$ such that
$D\ne R$ and $D\cap R\ne \emptyset$. Set $\{O\}:= D\cap R$. Let $U$ be the plane spanned by $D\cup R$.
Let $E\subset U$ be a general subset with cardinality $s-2w$. Let $A_1\subset D$ be the
zero-dimensional degree $w$ subscheme of $D$ with $O$ as its support. Fix $O'\in R\setminus \{O\}$ and call $A_2$ the zero-dimensional
subscheme of $R$ with $O'$ as its support and degree $w$. Set $A:= A_1\cup A_2$. There is
$P\in \langle \nu _d(A_1\cup A_2\cup E)\rangle$ such that
$br (P) = s$ and $sr (P) = 2d+3+s-4w$.
\end{lemma}

\begin{proof}
We will always compute the residual schemes with respect to divisors of $U$. Notice that $A:= A_1\cup A_2\cup E$ is curvilinear
and hence it only has finitely many subschemes. Hence there is $P\in \langle \nu _d(A_1\cup A_2\cup E)\rangle$ such that
$P\notin \langle \nu _d(F)\rangle$ for any $F\subsetneq A_1\cup A_2\cup E$. Since $\deg (A) =s
\le (d+1)/2$, we get $sr (P)=s$ and that $A$ is the only subscheme of $\mathbb {P}^m$ evincing $sb (P)$ (Remark \ref{v2}).

\quad (a) Fix $P\in \langle \nu _d(A_1\cup A_2\cup E)\rangle$ such that $P\notin \langle \nu _d(F)\rangle$ for any $F\subsetneq A_1\cup A_2\cup E$. In this step
we prove that $sr (P) \ge 2d+3+s-2w$. Assume $sr (P) \le 2d+2+s-2w$. By \cite{cs}, Proposition 3.1,
or \cite{lt}, subsection 3.2, there is $B\subset U$ evincing $sr (P)$. Since $A$ is not reduced, we have
$A\ne B$. Hence $h^1(U,\mathcal {I}_{A\cup B}(d)) >0$. Set $W:= A\cup B$. Since $\deg (W) \le 2d+2s-6 < 3d$, either there
is a line $L\subset U$ such that $\deg (L\cap W) \ge d+2$ or there is a conic $T \subset U$ such that $\deg (T\cap W)\ge 2d+2$ (Lemma \ref{c0}).

\quad (a1) Assume the existence of a line $L\subset U$ such that $\deg (L\cap W) \ge d+2$. If $h^1(U,\mathcal
{I}_{\mbox{Res}_L(W)}(d-1)) =0$, then Lemma \ref{v4} gives $A_1\cup A_2\subset L$, absurd. Hence
$h^1(U,\mathcal
{I}_{\mbox{Res}_L(W)}(d-1))>0$. Since $\deg (\mbox{Res}_L(W)) \le 2(d-1)+1$, there is a line $L'\subset U$
such that $\deg (L'\cap \mbox{Res}_L(W))\ge d+1$. Since $\deg (\mbox{Res}_{L\cup L'}(W)) \le d-1$, Lemma
\ref{v4} gives $A_1\cup A_2 \subset L\cup L'$ and $\mbox{Res}_{L\cup L'}(A)
= B\setminus B\cap (L\cup L')$. Since $w\ge 3$, Bezout theorem
gives $L\cup L' = D\cup R$. Hence $E\subset B$ and $B\setminus E \subset D\cup R$. Lemma \ref{v4} and Remark \ref{v4.0} applied to $D\cup R$ give 
the existence of $Q\in \langle \nu _d(A\cap (D\cup R))\rangle$ such that $A_1\cup A_2$ evinces
$br (Q)$, while $B\setminus E$ evinces $sr (Q)$. Since
$\sharp (B\setminus E)\le 2d+2-2w$, either $\sharp ((B\setminus E)\cap D) \le d+1-w$ or $\sharp ((B\setminus E)\cap R)
\le d-w$. First assume  $\sharp ((B\setminus E)\cap D) \le d+1-w$. Since $\deg (\mbox{Res}_R(A_1\cup A_2))
= w-1$, we have $h^1(\mathcal {I}_{\mbox{Res}_R((A\cup B)\cap (D\cup R))}(d-1))=0$.
Hence Lemma \ref{v4} applied to $A\cap (D\cap R)$ and $B\cap (D\cup R)$ gives
$\mbox{Res}_R(A_1\cup A_2) = \mbox{Res}_R(B\cap (D\cup R))$. Since $B$ is reduced,
we get $w\le 2$, a contradiction.

\quad (a2) Now we assume the existence of a conic $T \subset U$ such that $\deg (T\cap W)\ge 2d+2$. Since $\deg (\mbox{Res}_T(W)) \le d-1$, we have
$h^1(\mathcal {I}_{\mbox{Res}
_T(W)}(d-2)) =0$. Hence the case $t=2$ of Lemma \ref{v4} gives $A_1\cup A_2\subset T$
and $B\setminus B\cap T = \mbox{Res}_T(A)$. Since $w\ge 3$, and $A_1\cup A_2\subset T$,
Bezout theorem gives $T = D\cup R$. We work as in step (b1) (notice that in the case $L\cup L' = D\cup R$ we
only used that $\deg ((L\cup L')\cap (A\cup B)) \ge 2d+2$).

\quad (b) In this step we check the existence of $P\in \langle \nu _d(A_1\cup A_2\cup E)\rangle$ such that $P\notin \langle \nu _d(F)\rangle$ for any $F\subsetneq A_1\cup A_2\cup E$
and $sr (P) = 2d+3+s-2w$ and $br (P) =s$. Lemma \ref{o1} gives the existence of $O\in \langle \nu _d(A_1\cup A_2)\rangle$
such that $sr (O) \le 2d+3-2w$ and $O\notin \langle \nu _d(G)\rangle$ for any $G\subsetneq A_1\cup A_2$. Take a general $P\in \langle \{O\}\cup \nu _d(E)\rangle$.
Obviously $sr(P) \le sr (O) +\sharp (E) \le 2d+3+s-2w$. Step (b) gives $sr (P) = 2d+3+s-2w$. Assume $br (P) <s$. Since $br (P) \le d+1$, there is a zero-dimensional
scheme $W\subset \mathbb {P}^m$ such that $\deg (W)=br (P)$, $P\in \langle \nu _d(W)\rangle$ and $P\notin \langle \nu _d(W')\rangle$ for
any $W' \subsetneq W$. First assume $W\nsubseteq A_1\cup A_2\cup F$. Lemma \ref{v1} gives $h^1(\mathcal {I}_{A_1\cup A_2\cup E\cup W}(d)) >0$.
Since $bs (P) +s \le 2d+1$, there is a line $L\subset \mathbb {P}^m$ such that $\deg (L\cap (A_1\cup A_2\cup E\cup W)) \ge d+2$. As in step (b1) we get
a contradiction. Now assume $W\subsetneq A_1\cup A_2 \cup E$. Set $A':= (A_1\cup A_2)\cap W$. 
Lemma \ref{v3} gives $\langle \{P\}\cup \nu _d(E)\rangle \cap \langle \nu _d(A_1\cup A_2)\rangle =\{O\}$ and $\langle \{P\}\cup \nu _d(E')\rangle \cap \langle
\nu_d(A_1\cup A_2)\rangle$ . Since $P\in \langle \nu _d(W)\rangle$ and $P\notin \langle \nu _d(W')\rangle$ for
any $W' \subsetneq W$, the set $\langle \{P\}\cup \langle \nu _d(W\cap E)\rangle \cap \langle \nu _d(A')\rangle$, is a single point. Hence $E\subseteq W$
and $\{O\} \in \langle \nu _d(A')\rangle$, contradicting the choice of $O$.\end{proof}

Quoting Lemma \ref{o2} instead of Lemma \ref{o1} we get the following result.

\begin{lemma}\label{o4}
Fix integers $w\ge 2$ and $s\ge 2w+1$. Assume $d\ge 2s-1$. Fix lines $D, R\subset \mathbb {P}^m$ such that
$D\ne R$ and $D\cap R\ne \emptyset$. Set $\{O\}:= D\cap R$. Let $U$ be the plane spanned by $D\cup R$.
Let $E\subset U$ be a general subset with cardinality $s-2w-1$. Let $A_1\subset D$ be the
zero-dimensional degree $w+1$ subscheme of $D$ with $O$ as its support. Fix $O'\in R\setminus \{O\}$ and call $A_2$ the zero-dimensional
subscheme of $R$ with $O'$ as its support and degree $w$. Set $A:= A_1\cup A_2$. There is
$P\in \langle \nu _d(A_1\cup A_2\cup E)\rangle$ such that
$br (P) = s$ and $sr (P) = 2d+1+s-4w$.
\end{lemma}

\vspace{0.3cm}

\qquad {\emph {Proof of Theorem \ref{i1}.}} Since the cases $s=2,3$ are true by \cite{bgi}, Theorems 32 and 37,  we may assume $s\ge 4$ (the case $(s,r) =(3,2d-1)$
does not occur in the statement of Theorem \ref{i1}, because we assumed $r \le 2d+s-7$; this inequality is used in steps (d) and (e) below).

Notice that $ \sigma _s(X_{m,d})\setminus \sigma _{s-1}(X_{m,d}) \ne \emptyset$ (e.g., because  $s(m+1) < \binom{m+d}{m}$). Fix $P\in \sigma _s(X_{m,d})\setminus \sigma _{s-1}(X_{m,d})$ and
write $r:= sr (P)$. Since $P\notin \sigma _{s-1}(X_{m,d})$ we have $r \ge s$. Since $ \sigma _s(X_{m,d})\ne\sigma _{s-1}(X_{m,d})$ a non-empty open subset of $ \sigma _s(X_{m,d})$ is formed by points
with rank $s$. Hence to prove Theorem \ref{i1}  we may assume $r>s$. By Remark \ref{v2} there is a unique
degree $s$ zero-dimensional scheme $A\subset \mathbb {P}^m$ such that $P\in \langle \nu _d(A)\rangle$ and this scheme is smoothable. By Remark \ref{v2} there is no zero-dimensional scheme $A_1\subset \mathbb {P}^m$
such that $\deg (A_1)<s$ and $P\in \langle \nu _d(A_1)\rangle$. Hence $P$ has cactus rank $s$
and $P\notin \langle \nu _d(A')\rangle$ for any $A'\subsetneq A$.
Since $sr (P)=r$ there is a finite set $B\subset \mathbb {P}^m$ such that
$\sharp (B) =r$, $P\in \langle \nu _d(B)\rangle$ and $P\notin \langle \nu _d(B')\rangle$ for any $B'\subsetneq B$. Set $W:= A\cup B$. We have
$\deg (W)\le \deg (A)+\deg (B) =r+s$ and equality holds if and only if $A\cap B=\emptyset$.
Lemma \ref{v1} gives $h^1(\mathcal {I}_W(d))>0$.
Hence $\deg (W) \ge d+2$ (e.g., by \cite{bgi}, Lemma 34). Therefore $\sigma _{s,x}(X_{m,d})
= \emptyset$ if $s+1 \le x \le d-s+1$. 

We have $\sigma _{s,d-s+2}(X_{m,d}) \ne
\emptyset$, because $\sigma _{s,d-s+2}(X_{1,d}) \ne \emptyset$ by a theorem of Sylvester
(\cite{cs}, \cite{lt}, Theorem 4.1, or \cite{bgi}) and for any line $L\subset \mathbb {P}^m$
and any $P\in \langle \nu _d(L)\rangle$ the symmetric rank and the border rank of
$P$ are the same with respect to $X_{m,d}$ or with respect to $\nu _d(L) \cong X_{1,d}$
(\cite{lt}, subsection 3.2).

\quad (a) In this step we prove that $\sigma _{s,r}(X_{m,d}) \ne \emptyset$
for every $r\in \{d-s+3,\dots ,d+s-2\}$ such that $r+s \equiv d \pmod{2}$. Fix $r\in \{d-s+3,\cdots ,d+s-2\}$
such that $r+s \equiv d \pmod{2}$. Set $b:= (d+2+s-r)/2$. Since $r+s \equiv d \pmod{2}$, we have
$b\in \mathbb {Z}$. Since $r \le d+s-2$, we have $b \ge 2$. Since $r\ge d-s+3$, we have
$b < s$. Since $d \ge 2s-2$, we have $r>s$ and hence $2b < d+2$. 
Fix a line $L\subset \mathbb {P}^m$ and a connected zero-dimensional scheme $Z'\subset L$
such that $\deg (Z') = b$. Take any $Q\in \langle \nu _d(Z')\rangle$ such that
$Q\notin \langle \nu _d(Z'')\rangle$ for any $Z''\subsetneq Z'$ ($Q$ exists and the set of all such points $Q$
is a non-empty open subset of a projective space of dimension $\deg (Z')-1$, because
$Z'$ is a divisor of the smooth curve $L$).
Fix any set $E\subset \mathbb {P}^m\setminus L$ such that $\sharp (E) =s-b = (s+r-d-2)/2$. Since
$E\cap L= \emptyset$, we have $\deg (Z'\cup E) =s$. Since $d\ge s-1$, we
have $\dim (\langle \nu _d(Z'\cup E)\rangle )=s-1$. Since $L$ is contained in a smooth curve, $Z'$
is curvilinear. Since $E$ is a finite set, the scheme $Z'\cup E$ is curvilinear. We claim that any zero-dimensional curvilinear subscheme $W\subset \mathbb {P}^m$ has
only finitely many subschemes. Indeed, $W$ is contained in a smooth curve $C$ and hence we may write
$W = \sum _{i=1}^{x} a_iQ_i$ for some $x\in \mathbb {N}\setminus \{0\}$, $a_i\in \mathbb {N}\setminus \{0\}$ and $Q_i\in C$. The subschemes of $W$ are the
effective divisors of $C$ of the form $\sum _{i=1}^{x} b_iQ_i$ for some $b_i\in \{0,\dots ,a_i\}$. Hence $W$ has
exactly $\prod _{i=1}^{x} (a_i+1)$ subschemes. Hence $Z'\cup E$ has only finitely many closed subschemes.
Fix any $O\in \langle \nu _d(Z'\cup E)\rangle$ such that $O\notin \langle \nu _d(F)\rangle$
for any $F\subsetneq Z'\cup E$ ($O$ exists and the set of all such points $O$
is a non-empty open subset of the $(s-1)$-dimensional projective space $\langle \nu _d(Z'\cup E)\rangle$, because $Z'\cup E$ has only finitely many subschemes and $O\notin \langle \nu _d(L\cup E')\rangle$
for any $E'\subsetneq E$). Let $A'\subset L$ be a set evincing
$sr (Q)$ with respect to the rational normal curve $\nu _d(L)$. A theorem of Sylvester gives
$\sharp (A') =d-\deg (Z')+2$ (\cite{cs}, \cite{lt}, Theorem 4.1, or \cite{bgi}). Set $G:= A'\cup E$.
Since $E\cap L=\emptyset$, we have $\sharp (G)=r$.

\quad {\emph {Claim 1:}} $br (O) = s$.

\quad {\emph {Proof of Claim 1:}} We have $O\in \langle \nu _d(Z'\cup E)\rangle$ and
$O\notin \langle \nu _d(F)\rangle$
for any $F\subsetneq Z'\cup E$. Apply Remark \ref{v2}.

\quad {\emph {Claim 2:}} $sr (O) = r$ and $G$ evinces $sr (O)$.

\quad {\emph {Proof of Claim 2:}} Since $P\in \langle \nu _d(G)\rangle$, we have $sr (P)\le r$.
Since $s\le (d+1)/2$, $Z'\cup E$ is the only scheme evincing $br (O)$.
Since $Z'$ is not reduced, we get $sr (P) >s$. Fix any $U\subset \mathbb {P}^m$ evincing $sr (P)$. Since $sr (P) +br (P) \le 2d+1$, $sr (P) >br (P)$ and $Z'\cup E$ evinces
$br (O)$, \cite{bb1}, Theorem 1, gives
the existence of a line $D\subset \mathbb {P}^m$ such that $(Z'\cup E)\setminus D\cap (Z'\cup E)
= U\setminus U\cap D$ and every unreduced connected component of $Z'\cup E$ is contained in $D$.
Since $Z'$ is not reduced, we get $D=L$. Hence $U\setminus U\cap L = E$.
Since $ O\in \langle \nu _d(Z'\cup E)\rangle \subseteq \langle \nu _d(L\cup E)\rangle$, Lemma \ref{v3}
gives 
that $\langle \{O\}\cup \nu _d(E)\rangle \cap \langle \nu _d(L)\rangle$ is a single point (call it $Q'$).
Since $O\in \langle \{Q\}\cup \nu _d(E)\rangle$ and $Q\in \langle \nu _d(L)\rangle$, we have
$Q' = Q$. Lemma \ref{v3} gives $Q\in \langle \nu _d(L\cap U)\rangle$. Since $A'$ evinces
$sr (Q)$, we get $sr (P) = \sharp (U) \ge \sharp (E) +\sharp (A')$, concluding the proof of Claim 2.

The point $O$ shows that  $\sigma _{s,r}(X_{m,d}) \ne \emptyset$ for any $r\in \{d+3-s,\dots ,d+s-2\}$ such that $r+s \equiv d \pmod{2}$.

\quad (b) Fix any integer $r$ such that $r\in \{d-s+3,\dots ,d+s-2\}$ and $r+s \equiv d+1 \pmod{2}$.
In this step we prove that $\sigma _{s,r}(X_{m,d}) =\emptyset$. Assume
the existence of $P\in \sigma _{s,r}(X_{m,d})$. Fix $A\subset \mathbb {P}^m$ evincing
$br (P)$ and $B\subset \mathbb {P}^m$ evincing $sr (P)$. Since $r>s$ we have $A\ne B$.
As in step (a) we get the existence of a line $D\subset \mathbb {P}^m$
such that $\deg ((A\cup B)\cap D) \ge d+2$, every unreduced connected component of
$A$ is contained in $D$ and $\mbox{Res}_D(A) = B\setminus B\cap D$. Set $E:= B\setminus B\cap D$. By Lemma \ref{v3} the set $\langle \nu _d(D)\rangle \cap \langle \nu _d(E)\cup \{P\}\rangle$
is a unique point, $O$, and $O\in \langle \nu _d(A\cap D)\rangle \cap \langle \nu _d(B\cap D)\rangle$. Since $\deg (A) =\deg (A\cap D)+\sharp (E)$ (resp.
$\sharp (B) = \sharp (B\cap D)+\sharp (E)$), $A\cap D$ evinces $br (P)$
(resp. $B\cap D$ evinces $sr (O)$). The quoted theorem of Sylvester gives
$sr (O) +br (O) =d+2$. Hence $s + r = 2\cdot \sharp (E) +sr (O)+br (O) \equiv d \pmod{2}$, a contradiction.

\quad (c) In this step we fix an integer $r$ such that $d+s-1 \le r \le 2d+1-s$. In order to obtain a contradiction we assume $\sigma _{s,r}(X_{m,d}) \ne \emptyset$ and fix $P\in \sigma _{s,r}(X_{m,d})$. Take $A\subset \mathbb {P}^m$ evincing $br (P)$ and $B\subset \mathbb {P}^m$
evincing $sr (P)$. Let $A_1$ be the union of the connected components
of $A$ which are not reduced. Since $r>s$, we have $A_1\ne \emptyset$. Hence
$\deg (A_1) \ge 2$. Lemma \ref{v1} gives $h^1(\mathcal {I}_{A\cup B}(d)) >0$. Since
$\deg (A\cup B)\le s+r \le 2d+1$, there is a line $D\subset \mathbb {P}^m$ such
that $\deg ((A\cup B)\cap D) \ge d+2$ (Lemma \ref{c0} or \cite{bgi}, Lemma 34). Set $E:= B\setminus B\cap D$. Since
$e:= \deg (\mbox{Res}_D(A\cup B)) \le r+s-d-2 \le d$, we have
$h^1(\mathcal {I}_{\mbox{Res}_D(A\cup B)}(d-1))=0$. Hence Lemma \ref{v4} gives
$A_1 \subset D$ and $A\setminus A\cap D = E$, $e = \sharp (E)$ and
$e = \deg (A\setminus A\cap D) \le s-\deg (A_1)$. Since $A$ evinces $br (P)$,
the set $\langle \nu _d(E)\cup \{P\}\rangle \cap \langle \nu _d(A_1)\rangle$
is a unique point, $O$. Sylvester's theorem gives $sr (O)\le d$. Since $P\in \langle
\nu _d(E)\cup \{O\}\rangle$, we get $sr (P) \le e+d \le d+s-2$, a contradiction.

\quad (d) In this step we prove that $\sigma _{s,r}(X_{m,d}) \ne \emptyset$ for every integer $r$ such that
$2d+2-s \le r \le 2d+s-7$ and $r+s \equiv 0\pmod{2}$. Set $b:= (2d+2+s-r)/2$. Since $r+s\equiv 0 \pmod{2}$, we have
$b\in \mathbb {Z}$. Since $r \le 2d+s-7$, we have $b\ge 9/2$ and hence $b \ge 5$. Since $r\ge 2d+2-s$, we have $b \le s$. We may assume $m=2$
(\cite{lt}, subsection 3.2).
Fix a smooth conic $C\subset \mathbb {P}^2$, a connected zero-dimensional scheme $A_1\subset C$ such that
$\deg (A_1)=b$ and a general set $E\subset \mathbb {P}^2\setminus C$ such that $\sharp (E)=s-b$. We have $\sharp (E)\le d-1$.
Set $A:= A_1\cup E$ and $\{O'\}:= (A_1)_{red}$.
Fix $P\in \langle \nu _d(A)\rangle$ such that $P\notin \langle \nu _d(F)\rangle$ for any scheme $F\subsetneq A$ ($P$ exists,
because $A$ is curvilinear and $\dim (\langle \nu _d(A)\rangle )=\deg (A)-1$). Since $s=\deg (A) \le (d+1)/2$, then $br (P)=s$ and $A$ is
the only scheme evincing $br (P)$ (Remark \ref{v2}). Lemma \ref{v3.0} gives that the set $\langle \{P\}\cup E\rangle
\cap \langle \nu _d(C)\rangle$ is a unique point, $O$, that $O\in \langle \nu _d(A_1)\rangle$ and that
$sr (O) =b$. Since $b\le d+1$, the quoted theorem of Sylvester gives $r_{\nu _d(C)}(O) =2d+2-b$. Fix
$B_1\subset C$ such that $\nu _d(B_1)$ evinces $r_{\nu _d(C)}(O)$, i.e. take
$B_1\subset C$ such that $\sharp (B_1) =2d+2-b$, $O\in \langle \nu _d(B_1)\rangle$ and $O\notin \langle \nu _d(F)\rangle$
for any $F\subsetneq B_1$. We have $P\in \langle \nu _d(B_1\cup E)\rangle$. Lemma \ref{v3.0} also gives $P\notin \langle
\nu _d(G)\rangle$ for any $G\subsetneq B_1\cup E$. Since $br (P) =s$, $P\in \langle \nu _d(B_1\cup E)\rangle$ and
$\sharp (B_1\cup E) = r$, to prove that $\sigma _{s,r}(X_{m,d}) \ne \emptyset$ it
is sufficient to prove that $sr (P) \ge r$. Assume $sr (P)<r$ and take $B$ evincing $sr (P)$. We have $h^1(\mathcal {I}_{A\cup
B}(d))>0$ (Lemma \ref{v1}). Since $\deg (A\cup B) \le s+r-1<3d$, either there is a line $D\subset \mathbb {P}^2$ such that
$\deg (D\cap (A\cup B))\ge d+2$
or there is a conic $T$ such that $\deg (T\cap (A\cup B)) \ge 2d+2$ (Lemma \ref{c0}). 

\quad (d1) Here we assume the existence of a line
$D \subset \mathbb {P}^2$ such that $\deg (D\cap (A\cup B)) \ge d+2$. If $h^1(\mathcal {I}_{\mbox{Res}_D(A\cup B)}(d-1))=0$,
then Lemma \ref{v4} gives $A_1\subset D$, a contradiction. Hence $h^1(\mathcal {I}_{\mbox{Res}_D(A\cup B)}(d-1))>0$. Since
$\deg (\mbox{Res}_D(A\cup B))\le r+s-1-d-2\le 2(d-1)+1$, Lemma \ref{c0} or \cite{bgi}, Lemma 34, give the existence
of a line $D'\subset \mathbb {P}^2$ such that $\deg (D'\cap \mbox{Res}_D(A\cup B))\ge d+1$.
Hence $\deg ((A\cup B)\cap (D\cup D')) \ge 2d+3$. Hence 
$\deg (\mbox{Res}_{D\cup D'}(A\cup B)) \le d-1$.
Hence $h^1(\mathcal {I}_{\mbox{Res}_{D\cup D'}(A\cup B)}(d-2))=0$.
Lemma \ref{v4} gives $A_1\subset D\cup D'$. Since $C$ is an irreducible conic containing $A_1$, Bezout
theorem gives $b \le 4$, a contradiction.

\quad (d2) Now assume the existence of a conic $T\subset \mathbb {P}^2$ such that $\deg (T\cap (A\cup B)) \ge 2d+2$.
Since $\deg (\mbox{Res}_T(A\cup B)) \le d-1$, we have $h^1(\mathcal {I}_{\mbox{Res}_T(A\cup B)}(d-2)) =0$. Hence the case $t=2$ of Lemma \ref{v4} gives $A_1\subset T$ and $\mbox{Res}_T(A) = B\setminus B\cap T$. Since $b \ge 5$, $C$ is irreducible and $A_1\subset C$, Bezout theorem
gives $T=C$. We get $B\setminus B\cap C =E$. Hence
$\sharp (B\cap C) \le 2d+1-b$. Hence $\deg (C\cap (A\cup B)) \le 2d+1$, a contradiction.

\quad (e) In this step we prove that $\sigma _{s,r}(X_{m,d}) \ne \emptyset$ for every integer $r$ such that
$2d+3-s \le r \le 2d+s-7$ and $r+s \equiv 1\pmod{2}$ and hence conclude the proof of Theorem \ref{i1}. Set $c:= (2d+3+s-r)/2 $ and $w:= \lfloor c/2\rfloor$.
Since $r+s \equiv 1\pmod{2}$, we have $c\in \mathbb {Z}$. Since
$r \ge 2d+3-s$, we have $c \le s$. Since $r \le 2d+s-7$, we have $c\ge 5$. If $c$ is odd, then we apply Lemma \ref{o4}. If $c$ is even, then we apply Lemma \ref{o3}.\qed

\providecommand{\bysame}{\leavevmode\hbox to3em{\hrulefill}\thinspace}


\begin{thebibliography}{99}

\bibitem{bb} E. Ballico, A. Bernardi, Stratification of the fourth  secant variety of Veronese variety via the symmetric rank. arXiv.org/abs/1005.3465v3~[math.AG].

\bibitem{bb2} E. Ballico, A. Bernardi, A partial stratification of secant varieties of Veronese varieties via curvilinear subschemes. arXiv:1010.3546v3, Sarajevo J. Math. 8 (20), (2012),
33--52.

\bibitem{bb1} E. Ballico, A. Bernardi, Decomposition of homogeneous polynomials with low rank. Math. Z. 271 (2012), 1141--1149; DOI 10.1007/s00209-011-0907-6.

\bibitem{bgi} A. Bernardi, A. Gimigliano, M. Id\`{a},
Computing symmetric rank for symmetric tensors.
J. Symbolic. Comput. 46 (2011), 34--55.

\bibitem{bk} A. Bernardi, K. Ranestad, The cactus rank of cubic forms. J. Symbolic. Comput. 50 (2013) 291--297.
DOI: 10.1016/j.jsc.2012.08.001

\bibitem{bcmt} J. Brachat, P. Comon, B. Mourrain, E. P. Tsigaridas,  Symmetric tensor decomposition. Linear Algebra Appl. 433 (2010), no. 11--12, 1851--1872.


\bibitem{bb+} W. Buczy\'{n}ska, J. Buczy\'{n}ski, Secant varieties to high degree veronese reembeddings, catalecticant matrices and smoothable Gorenstein schemes. arXiv:1012.3562v4~[math.AG], J. Algebraic Geom. (to appear).


\bibitem{bgl} J. Buczy\'{n}ski, A. Ginensky, J. M. Landsberg, 
Determinantal equations for secant varieties and the 
Eisenbud-Koh-Stillman
conjecture. arXiv:1007.0192v4, Journal of London Mathematical Society (to appear)..

\bibitem{cs} G. Comas, M. Seiguer, On the rank of a binary form. Found. Comp. Math. 11 (2011), no. 1, 65--78.

\bibitem{cglm} P. Comon, G. H. Golub, L.-H. Lim, B. Mourrain, Symmetric tensors and symmetric tensor rank. SIAM J.  Matrix Anal.  30 (2008) 1254--1279.

\bibitem{cm} P. Comon, B. Mourrain, Decomposition of quantics in sums of powers of linear forms. Signal Processing, Elsevier 53, 2, 1996.

\bibitem{c2} A. Couvreur,  The dual minimum distance of arbitrary dimensional algebraic-geometric codes. J. Algebra 350 (2012), no. 1, 84--107.


\bibitem{ep} Ph. Ellia, Ch. Peskine, Groupes de points de ${\bf {P}}^2$: caract\`{e}re et position uniforme. Algebraic geometry (L'Aquila, 1988), 111--116,
Lecture Notes in Math., 1417, Springer, Berlin, 1990.

\bibitem{l} J. M. Landsberg, Tensors: Geometry and Applications.
Graduate Studies in Mathematics, Vol. 128, Amer. Math. Soc. Providence, 2012.

\bibitem{lt} J. M. Landsberg, Z. Teitler, On the ranks and border 
ranks of symmetric tensors. Found. Comput. Math. 10, (2010) no. 3, 339--366.

\bibitem{ls} L.-H. Lim, V. De Silva,
Tensor rank and the ill-posedness of the best low-rank approximation problem.
SIAM J. Matrix Anal. 30 (2008), no. 3, 1084--1127.






\end{thebibliography}
\end{document}